\documentclass[twoside,a4paper,reqno,11pt]{amsart} 
\usepackage[top=28mm,right=28mm,bottom=28mm,left=28mm]{geometry}

\usepackage{amsfonts, amsmath, amssymb, mathrsfs, bm, latexsym, stmaryrd, array, mathtools, bold-extra, enumerate}
\usepackage{mathrsfs}

\newcommand{\class}{{\mathscr{C}}}

\renewcommand{\to}{\rightarrow}

\makeatletter
\newcommand{\imod}[1]{\allowbreak\mkern4mu({\operator@font mod}\,\,#1)}
\makeatother

\newtheorem{thm}{Theorem}[section] 
\newtheorem*{conj*}{Conjecture}
\newtheorem{lemma}[thm]{Lemma}

\newtheorem{corol}[thm]{Corollary}

\newtheorem{prop}[thm]{Proposition}

\theoremstyle{definition}

\begin{document}

 \author{Rachel Pengelly}
 \address{Alan Turing Building, University of Manchester, Manchester, M13 9PL, UK}
 \email{rachel.pengelly@manchester.ac.uk}

\author{Donna M. Testerman}
\address{ Ecole Polytechnique F\'ed\'eral de Lausanne,
Institute of Mathematics, Station 8, CH-1015 Lausanne, Switzerland}
\email{donna.testerman@epfl.ch}
 
\thanks{R. Pengelly gratefully acknowledges the finanical support of the EPSRC grant EP/R513167/1, the LMS ECR Travel grant with reference ECR-2122-25, and EPFL for their hospitality.
D. Testerman acknowledges the support of the Swiss National Science Foundation via grant number $200020\textunderscore207730$.}

\title[Unicity of $A_1$-subgroups]{Unicity of $A_1$-subgroups associated to unipotent elements in simple algebraic groups}

\begin{abstract}

Let $k$ be an algebraically closed field of positive characteristic, and $G$ a simple algebraic group over $k$. 
Under the assumption that the characteristic is a good prime for $G$, we determine a maximal $G$-stable subvariety $\mathcal{U}'$ of the variety of the unipotent elements in $G$ such that for all $u \in \mathcal{U}'$ any two $A_1$-subgroups of $G$ containing $u$ are $G$-conjugate. 
 This result establishes to what degree an analogue of the Jacobson--Morozov Theorem for Lie algebras is valid for simple algebraic groups defined over fields of (good) positive characteristic. 
\end{abstract}

\date{\today}

\maketitle

\section{Introduction}\label{sec:intro}

Let $k$ be an algebraically closed field, and $G$ a simple algebraic group defined over $k$.
The aim in this paper is to  determine a maximal $G$-stable subvariety ${\mathcal U}'$ of the variety of unipotent elements in $G$ such that, for all
$u\in{\mathcal U}'$, any two
${\rm (P)SL}_2(k)$-subgroups of $G$ containing $u$ are $G$-conjugate. 
Hereafter we shall refer to a closed connected simple subgroup of $G$ of type $A_1$ as an
\emph{$A_1$-subgroup} of $G$ and to the $G$-conjugacy class of such subgroups as simply the \emph{class of $A_1$-subgroups}.

The above question  was motivated by recent progress on the analogous question concerning nilpotent orbits in Lie algebras. Set $\mathfrak{g} = \text{Lie}(G)$.
Assuming ${\rm char}(k)\ne 2$, an $\mathfrak{sl}_2$-triple in $\mathfrak{g}$ is a triple of elements of $\mathfrak g$ which generate a Lie subalgebra
isomorphic to $\mathfrak{sl}_2(k)$. In characteristic 0 the theorems of Morozov, Jacobson and Kostant (see \cite{Morozov, Jacobson, Kostant}) show that
any non-zero nilpotent
$e \in \mathfrak g$ can be embedded in an $\mathfrak{sl}_2$-triple $(e,h,f)$ in $\mathfrak{g}$, which is unique up to
the action of the centralizer of $e$ in $G$. A result of Pommerening, \cite{PomII}, extends this embedding to the odd positive good characteristic case. By \cite[Theorem 1.1]{StewTho}, this embedding is unique if and only if ${\rm char}(k) > h(G)$ the Coxeter number of $G$.  
For ${\rm char}(k) \leq h(G)$, this embedding fails to be unique. However, results of Goodwin--Pengelly recover uniqueness for a certain explicitly given maximal subvariety. 
This result can be found for classical
Lie algebras in \cite[Theorem 1.1]{GP}, with the corresponding result for exceptional Lie algebras in preparation. 

We return to our consideration of unipotent elements in $G$. If ${\rm char}(k) = 0$, for any non-identity unipotent element $u\in G$, any two $A_1$-subgroups containing $u$ are $G$-conjugate (this follows from the result for $\mathfrak{g}$). Hence we will
assume hereafter that $k$ is of characteristic $p>0$, and furthermore will assume that $p$ is a good prime for $G$. (See the end of the introduction for some remarks about this additional assumption.)
Note that even for $p$
arbitrarily large, there exist non-identity unipotent elements $u\in G$ lying in infinitely many non-conjugate $A_1$-subgroups of $G$; one can see this by
considering the representation theory of $A_1$, e.g. when $G$ is a classical type group. Nevertheless, for certain classes of unipotent elements of $G$, or for $A_1$-subgroups satisfying additional criteria, uniqueness results are available.
For example, the main result of \cite{LS_root} implies the unicity (up to $G$-conjugacy) of $A_1$-subgroups containing a long root element. In  \cite{SeitzgoodA1},
Seitz considers so-called ``good'' $A_1$-subgroups, that is $A_1$-subgroups $X \subseteq G$ such that the weights of the maximal torus of $X$ on
$\mathfrak{g}$ are less than or equal to $2p-2$. Seitz establishes the existence and unicity (up to $G$-conjugacy) of a good $A_1$-subgroup containing $u$, for any unipotent $u\in G$ of order $p$.

We now state our result for the classical type groups.

\begin{thm} \label{t:mainclassical}
Let $G$ be a simply connected simple algebraic group of classical type defined over an algebraically closed field $k$ of good characteristic $p>0$.  Let $u \in G$ be unipotent of order $p$. Then  any two $A_1$-subgroups containing $u$ are $G$-conjugate if and only if the Jordan block decomposition of $u$ acting on the natural module of $G$ is as given in Table \ref{Table1.1}.

\begin{table}[ht!]
\centering
\begin{tabular}{| c | c  c |}  
 \hline
  $G$ & \multicolumn{2}{|c|}{ Jordan block decomposition of $u$}  \\ 
 \hline\hline
 $A_n$ & $(J_\ell+J_1^r)$ & $\ell \leq p$ and if $\ell \in \{3,p\}$ then $r=0$ \\
 $C_n$ &  $(J_{a}+J_1^r)$ & $a$ even, $a<p$ \\
  & $(J_a^2+J_1^r)$ &  $a$ odd, $a<p$ \\
 $B_n$ or $D_n$ & $(J_a+J_1^r)$ & $a$ odd, $a\leq p$, and
  if $a = 3$ then $r = 0$ \\
 & $(J_a^2+J_1^r)$ & $a$ even, $a<p$ \\
 \hline
\end{tabular}
\label{Table1.1}
\end{table}   
\end{thm}

Note that the property of any two $A_1$-subgroups being $G$-conjugate is preserved under isogeny, hence it is enough to prove Theorem~\ref{t:mainclassical} for simply-connected classical type simple algebraic groups.

We contrast this result with the results in $\mathfrak{g}$, \cite[Theorem 1.1]{GP}.  There the maximal $G$-stable closed subvariety $\mathcal{V}\subseteq \mathfrak{g}$ consisting of nilpotent elements
such that the $G$-orbits of $\mathcal{V}$ are in bijection with the $G$-orbits of subalgebras of $\mathfrak{g}$ isomorphic
to $\mathfrak{sl}_2(k)$ and whose nilpotent elements lie in $\mathcal{V}$ is also described by the Jordan block decomposition of the action of the nilpotent elements on the natural module for $G$.
The variety $\mathcal{V}$ consists of those elements with Jordan block decomposition $(J_{a_1}+\dots + J_{a_k})$ such that either:
\begin{enumerate}[(a)]
    \item $G$ of type $A$ or $C$:  $a_i < p$ for $1 \leq i \leq k$; or 
    \item $G$  of type $B$ or $D$: $a_1 \leq p, a_i <p$ for $ 2 \leq i \leq k $.
\end{enumerate}
Hence we see that the conditions required for the unipotent elements to lie in a unique  class of $A_1$-subgroups are far more restrictive than  the corresponding conditions for
nilpotent elements in $\mathfrak{g}$.

For $G$ an exceptional type group, we will identify the class $\class(u)$ of a non-identity unipotent element $u\in G$  via its Bala--Carter--Pommerening label. An explanation of this notation can be found in \S\ref{sec:exc}. 

\begin{thm} \label{t:mainexceptional}
  Let $G$ be an exceptional type simple algebraic group defined over an algebraically closed field $k$ of good characteristic $p>0$.
  Let $u \in G$ be unipotent of order $p$. Then  any two $A_1$-subgroups containing $u$ are $G$-conjugate if and only if one of the following holds:\begin{enumerate}
  \item the $G$-conjugacy class of $u$ is $A_1,A_3$ or $D_4$;
  \item $u$ is a regular unipotent element in $G$; or
    \item $(G,p,\class(u))$
      is as in Table $\ref{Tablemain}$.
      \end{enumerate}

    \begin{table}[ht!]
\centering
\begin{tabular}{|  c | c | c |}  
 \hline
  $G$ & $p$ & $\class(u)$ \\ 
 \hline\hline 
 \rule{0pt}{3ex}
 $G_2$ & $\geq 5$ & $\tilde{A_1}$  \\
 $F_4$ & $\geq 5$ & $\tilde{A_2}, B_2, B_3, C_3, F_4(a_1)$  \\
 $E_6$ & $\geq 7$ & $A_5, D_5, E_6(a_1)$  \\
 $E_7$ & $7$ & $(A_5)'', (A_5)'$  \\
  & $\geq 11$ & $(A_5)'', (A_5)', D_5,A_6, D_6, E_6(a_1), E_6, E_7(a_1)$ \\
 $E_8$ & $7$ & $A_5$  \\
  & $\geq 11$ & $A_5, D_5, E_6(a_1), D_6, E_6, A_7, D_7, E_7(a_1), E_7, E_8(a_4), E_8(a_2), E_8(a_1)$  \\
 \hline
\end{tabular}
\label{Tablemain}
\end{table} 
\end{thm}

Without describing in detail the result for nilpotent orbits, we mention that here again the list of $G$-orbits of nilpotent elements for which there exists a
unique $G$-orbit of associated subalgebras isomorphic to ${\mathfrak{sl}}_2(k)$ is much longer than the above list of unipotent classes in the exceptional groups.

The results of \cite{SeitzgoodA1} allow us to give more detail on the properties of the $A_1$-subgroups containing $u$ of the form given in Theorems \ref{t:mainclassical} and \ref{t:mainexceptional}. 

\begin{corol}
Let $G$ be a simple algebraic group defined over an algebraically closed field $k$ of good characteristic $p>0$. Let $u \in G$ be unipotent of order $p$ such that any two $A_1$-subgroups containing $u$ are $G$-conjugate. 
Then any $A_1$-subgroup of $G$ containing $u$ is a good $A_1$-subgroup, and moreover any two such $A_1$-subgroups are conjugate by an element of $R_u(C_G(u))$.
\end{corol}

\begin{proof}
    This follows by Theorems \ref{t:mainclassical} and \ref{t:mainexceptional} in combination with \cite[Theorem 1.1]{SeitzgoodA1}.
\end{proof}

The main reference, until recent years, on classes of $A_1$-subgroups of the exceptional algebraic groups was \cite{LaTeA1}, where under certain restrictions on
$p$, the authors determine all  classes of $A_1$-subgroups and group them according to the conjugacy class of their non-identity unipotent elements.
Thus, Theorem~\ref{t:mainexceptional} follows directly from {\it loc.cit.} for large enough $p$. The proof of Theorem~\ref{t:mainexceptional}
for smaller primes relies on the results in \cite{Law_unimax}, where the author determines the $G$-fusion in an exceptional group $G$ of unipotent classes
represented in  maximal closed connected reductive subgroups of $G$, and on the more recent work of Thomas and Litterick--Thomas (see \cite{Thomas_A1, LiTh_CR}) providing
detailed information about simple subgroups of the  exceptional algebraic groups in small characteristics.

Finally, let us say something about the situation when $p$ is a bad prime for $G$. First we note that for the group $G_2$
defined over a field of characteristic $3$, there is a class of elements of order $3$ which does not intersect any  $A_1$-subgroup
of $G$ (see \cite{PST}). Nevertheless, one could still ask for which classes of elements of order $p$ a representative of the class meets at most
one class of $A_1$-subgroups. However, the absence of a classification of the non-$G$-completely reducible $A_1$-subgroups of the exceptional algebraic groups in bad characteristic (see \S\ref{sec:exc} for the relevant definitions) increases significantly the complexity of the analysis and we do not treat this case here.

{\it {Acknowledgements:}} The authors would like to thank David Craven for pointing out an error  in an earlier version of Theorem~\ref{t:mainclassical} and suggesting a shorter argument for parts of the proof. In addition, we thank Gunter Malle, Adam Thomas and Martin Liebeck for their useful comments, which we have incorporated in this version.

\section{Notation and Preliminaries} \label{prelims}

We assume throughout this paper that $k$ is an algebraically closed field of characteristic $p>0$. Let $G$ be a simple algebraic group over $k$ with $p$ a good prime for $G$. Let $F:G\to G$ be the standard Frobenius endomorphism of $G$. All representations
of algebraic groups are assumed to be finite-dimensional rational representations over $k$.

Let $L(c)$ denote the unique (up to isomorphism) irreducible $k{\rm SL}_2(k)$-module with highest weight $c \in \mathbb{Z}_{\geq 0}$.
Recall that for $0 \leq c \leq p-1 $ the non-identity unipotent elements
in ${\rm SL}_2(k)$ act with a single Jordan block of size $c+1$ on $L(c)$.
Also recall that ${\rm SL}_2(k)$ preserves a non-degenerate symplectic form on $L(c)$ if $c$ is odd, and a non-degenerate quadratic form if $c$ is even. 

Let $W(c)$ denote the $(c+1)$-dimensional Weyl module of highest weight $c \in \mathbb{Z}_{\geq 0}$ for ${\rm SL}_2(k)$. For $p \leq c \leq 2p-2 $, the non-identity unipotent elements in
${\rm SL}_2(k)$ act with Jordan blocks of sizes $p$ and $c-p+1$ on $W(c)$, see \cite[Proposition 5.7]{K2}.
Let $T(c)$ denote the indecomposable tilting module of highest weight $c \in \mathbb{Z}_{\geq 0}$. For $p\leq c\leq 2p-2$, recall that $\dim (T(c)) = 2p$ and the
non-identity unipotent elements in ${\rm SL}_2(k)$ act with 2 Jordan blocks of size $p$, see \cite[Lemma 5.3]{K2}. For larger $c$, we have a tensor product
decomposition of $T(c)$ as determined in
\cite[Lemma 5]{Erdmann_Henke}.

Let $V$ be a finite-dimensional $k$-vector space, and $u \in {\rm GL}(V)$ of order $p$. The action of $u$ on $V$ can be described in terms
of $k [u]$-modules. There exist exactly $p$ indecomposable $k [u]$-modules which we will
denote by $J(1), \dots, J(p)$ where the dimension of $J(c)$ is $c$, and $u$ acts on $J(c)$ as a full Jordan block. See the preliminaries of
\cite{Renaud} for more detail on the structure of these modules.

The following three results will allow us to reduce our considerations to a very short list of unipotent classes in $G$.

\begin{prop}\label{nonsimple}  Let $u\in G$ be of order $p$, with $F(u) = u$ and $u\in Y = Y_1Y_2$ with $[Y_1,Y_2] = 1$, $Y_i$ a non-trivial semisimple $F$-invariant
  closed subgroup of $G$, and $u\not\in Y_i$ for $i=1,2$.
  Then $u$ lies in non-conjugate $A_1$-subgroups of $G$.
\end{prop}

\begin{proof} Write $u = u_1u_2$ for $u_i\in Y_i\setminus\{1\}$. By \cite{TestA1}, for $i=1,2$, there exist  closed subgroups $X_i\subseteq Y_i$, of
  type $A_1$ with $u_i\in X_i$. Moreover,
  we may choose $X_i$ to be $F$-invariant and such that $F(u_i) = u_i$ for $i=1,2$. Fix separable isogenies $\varphi_i:{\rm SL}_2(k)\to X_i$, with $\varphi_i(\begin{pmatrix} 1&1\cr 0&1\end{pmatrix}) = u_i$. Let $Z_j$
  be the image of the map ${\rm SL}_2(k)\to X_1X_2$ given by $x\mapsto \varphi_1(x)\cdot F^j((\varphi_2(x))$, for $x\in {\rm SL}_2(k)$. Then considering
  the action of $Z_0$ and $Z_1$ on ${\rm Lie}(G)$, we see that $Z_0$ and $Z_1$ are non-conjugate $A_1$-subgroups of $G$ with $u\in Z_0\cap Z_1$.\end{proof}

\begin{prop} \label{p:longroot} Let $G$ be  a simple algebraic group and $u\in G$ a long root element of $G$. Then any two $A_1$-subgroups containing $u$ are $G$-conjugate.
\end{prop}

\begin{proof} This follows from \cite[Theorem 2.1]{LS_root}. \end{proof}

\begin{prop}\cite[Proposition 2.11]{BT_reg} \label{p:regular}
Let $G$ be  a simple algebraic group and $u\in G$ a regular unipotent element of $G$ of order $p$. Then any two $A_1$-subgroups containing $u$ are $G$-conjugate.
\end{prop}

\section{Classical Groups}

Let $G$ be a simple algebraic group of classical type. Given the assumption that $p$ is good for $G$, we are able to parametrise the unipotent classes via partitions. A given partition corresponds to the
Jordan blocks of the element in the natural representation of the classical group, and with the exception of type $D_n$, a given partition corresponds to at most one class of unipotent elements in the classical group. We write the Jordan blocks corresponding to the action of unipotent elements in the form $(J_k^{r_k} + \dots + J_1^{r_1})$ where $r_i \geq 0$ for all $i$.

In the following three lemmas, let  $\langle u\rangle$ be a cyclic group of order $p$; we consider the decompositions of tensor products of indecomposable $k[u]$-modules.

\begin{lemma}\label{tensors_induction}  Let $2 \leq i, j \leq p$ and suppose $(J(i-1)\otimes J(j))\downarrow {k[u]} = \oplus_{r=1}^t J(n_r)$ for some $n_r\leq p$. Then for some $k[u]$-module $N$, we have
  $J(i)\otimes J(j) = \oplus_{r=1}^t J(\ell_r)\oplus N$, where $\ell_r\geq n_r$ for all $1\leq r\leq t$.
\end{lemma}

\begin{proof} We have that $J(i-1)$ is a $k[u]$-submodule of $J(i)$. So $S:=J(i-1)\otimes J(j)$ is a  $k[u]$-submodule of $J(i)\otimes J(j)$. Moreover,
  $S^u\subseteq (J(i)\otimes J(j))^u$ and $(u-1)^m(S)\subseteq (u-1)^m(J(i)\otimes J(j))$ for all $m\geq 1$. These two facts give the result.
\end{proof}

\begin{lemma}\label{twofold_tensors} Let $1<m\leq n \leq p$.  If $(m,n) \not\in \{ (2,2),(2,p) \}$, the direct sum decomposition of $J(m) \otimes J(n)$ into indecomposable $k[u]$-modules contains either:
\begin{enumerate}[(a)]
    \item at least three non-trivial indecomposable summands; or
    \item two non-trivial indecomposable summands of distinct dimension.
\end{enumerate}
Moreover, $J(2) \otimes J(2) \cong J(3) \oplus J(1)$ if $p \neq 2$, and $J(2) \otimes J(p) \cong (J(p))^2$. 
\end{lemma}

\begin{proof}
Calculations of the decomposition of tensor products follow from the formulae in \cite[Theorem 1]{Renaud}.
  
First suppose that $m=2$, then $J(2) \otimes J(n)$ is given by $(J(p))^2$ if $n=p$, and $J(n+1) \oplus J(n-1)$ if $n < p$. Hence if $2< n<p$, then $J(2) \otimes J(n)$ contains two non-trivial indecomposable summands of distinct dimension.

Now suppose that $m=n=3$, and note that it follows that $p>2$. In this case we have that $J(3) \otimes J(3)$ is given by $(J(3))^3$ if $p=3$, and $J(5) \oplus J(3) \oplus J(1)$ if $p\geq5$. Hence $J(3) \otimes J(3)$ contains either three non-trivial indecomposable summands, or two non-trivial indecomposable summands of distinct dimension. 

Next suppose that $m=3,n=4$, and note that it follows that $p\geq 5$. In this case we have that $J(3) \otimes J(4)$ is given by $J(5)^2 \oplus J(2)$ if $p=5$, and $J(6) \oplus J(4) \oplus J(2)$ if $p>5$. Hence the direct sum decomposition of $J(3) \otimes J(4)$ contains three non-trivial indecomposable summands.
For $m \geq 3, n \geq 4$, it follows from repeated applications of Lemma \ref{tensors_induction} that $J(m) \otimes J(n)$ contains at least three non-trivial indecomposable summands.
\end{proof}

\begin{lemma}\label{threefold_tensors} For all $t \geq 3$, and $2 \leq m_1 \leq \dots \leq m_t \leq p$, the direct sum decomposition of $J(m_1) \otimes J(m_2) \otimes \dots \otimes J(m_t)$ into indecomposable $k[u]$-modules contains at least three non-trivial indecomposable summands. 
\end{lemma}

\begin{proof}
We first consider the case $t = 3$.  Let $m_1=m_2=m_3=2$, then $J(2) \otimes J(2) \otimes J(2)$ is equal to $(J(2))^4$ if $p=2$, $(J(3))^2\oplus J(2)$ if $p=3$, and $J(4)\oplus(J(2))^2$ if $p>3$.
In each case, the direct sum decomposition contains at least three non-trivial indecomposable summands. 
Now suppose that $2 \leq m_1 \leq m_2 <m_3$, then it follows from repeated applications of Lemma \ref{tensors_induction} that the direct sum decomposition of $J(m_1) \otimes J(m_2) \otimes J(m_3)$ contains at least three non-trivial indecomposable summands. 

We proceed inductively on $t$, and suppose now that $t>3$. Suppose that the direct sum decomposition of any $n$-fold tensor product where $n<t$ contains at least three non-trivial indecomposable summands. 
We consider $S:= J(m_1) \otimes J(m_2) \otimes \dots \otimes J(m_{t})$. Set $S':= J(m_1) \otimes J(m_2) \otimes \dots \otimes J(m_{t-1})$, so that $S = S' \otimes J(m_t)$. We have that $S'$ is equal to $J(k_1)\oplus J(k_2) \oplus \dots \oplus J(k_l)$ for some $k_1,k_2,k_3>1$, $l \geq 3$. 
We have that $S$ is given by $(J(k_1)\otimes J(m_t)) \oplus \dots \oplus (J(k_l)\otimes J(m_t))$. By Lemma \ref{twofold_tensors} each term $(J(k_i) \otimes J(m_t))$ contributes at least one non-trivial indecomposable summand for $i \in \{1,2,3\}$. Hence the direct sum decomposition of $S$ contains at least three non-trivial indecomposable summands.
\end{proof}

We can now establish Theorem \ref{t:mainclassical} for $G$ of type $A_n$.

\begin{prop} \label{prop:SLn}
  Let $G $ be $\rm{SL}(W)$ and let $u\in G$ be unipotent of order $p$. Then any two $A_1$-subgroups containing $u$ are $G$-conjugate if and only if 
     $u$ acts on $W$ with Jordan blocks $(J_{\ell}+J_1^r)$, for $\ell \leq p$, where if $\ell \in \{3,p\}$ then $r = 0$. 
\end{prop}

\begin{proof} 
Using Proposition~\ref{nonsimple}, we reduce to elements having exactly one non-trivial Jordan block.

If $u$ acts on $W$ as a single Jordan block, then $u$ is regular in $G$, and the result follows by Proposition \ref{p:regular}.

Suppose that $u$ acts on $W$ with Jordan blocks $(J_{p}+J_1^r)$ for $r>0$. In this case we exhibit two non-conjugate $A_1$-subgroups of $G$ containing $u$.
The first of these can be obtained by the action of ${\rm SL}_2(k)$ on $L(p-1) \oplus N$, where $N$ is an $r$-dimensional trivial module. 
It follows from the block structure of $u$ on $W(p)$ that the action of ${\rm SL}_2(k)$ on $W(p) \oplus N'$, where $N'$ is an $(r-1)$-dimensional trivial module gives rise to a second class of $A_1$-subgroups containing $u$. 

Suppose now that $u$ acts on $W$ with Jordan blocks $(J_{3}+J_1^r)$  for $r>0$ where we may now assume that $p>3$. We exhibit two non-conjugate $A_1$-subgroups of $G$ containing $u$. 
The first of these can be obtained as in the previous paragraph, via the action of ${\rm SL}_2(k)$ on $L(2) \oplus N$, where $N$ is an $r$-dimensional trivial module. 
The second action is that of ${\rm SL}_2(k)$ on $(L(1) \otimes L(1)^{F}) \oplus N'$, where $N'$ is an $(r-1)$-dimensional trivial module  (see Lemma~\ref{twofold_tensors}).

Finally suppose that $u$ acts on $W$ with Jordan blocks $(J_{\ell}+J_1^r)$, $\ell \not \in \{3,p\}$. 
If $u$ lies in an $A_1$-subgroup $X$ of $G$, then by \cite[Corollary 5.8]{K2} we have that $X$ acts completely reducibly on $V$.
Lemmas \ref{twofold_tensors} and \ref{threefold_tensors} show that there is a unique (up to Frobenius twists and equivalence) $kX$-module such that 
the action of $u$ is given by $(J_{\ell}+J_1^r)$, $\ell \not \in \{3,p\}$. This is given by the action of ${\rm SL}_2(k)$ on $L(\ell-1) \oplus N$, 
where $N$ is an $r$-dimensional trivial module. 
\end{proof}

We now consider the classical groups of type $B_n,C_n$ and $D_n$, and recall that we are assuming that $p>2$. In addition, a unipotent element $u \in \mathrm{GL}(W)$ lies in the symplectic group ${\rm Sp}(W)$
if and only if each odd size
Jordan block occurs an even number of times and a unipotent element lies in the orthogonal group ${\rm SO}(W)$ if and only if each even size Jordan block occurs an even
number of times.

Let $X = {\rm SL}_2(k)$. Then given any homomorphism $\rho:X\to {\rm Sp}(W)$, we have that $W\downarrow X$ is an orthogonal direct
sum of irreducible modules, on each of  which $X$ preserves a non-degenerate symplectic form, and  modules of the form $M\oplus M^*$, for $M$ any
irreducible $kX$-module.
Moreover, $\rho(X)\subseteq {\rm Sp}_{2r_1}(k)\times\cdots\times {\rm Sp}_{2r_t}(k)$, each factor corresponding either to an irreducible summand or to
one of the modules $M\oplus M^*$. Now let $u\in{\rm Sp}(W)$ be a unipotent element. Applying Proposition~\ref{nonsimple}, we see that if all $A_1$-subgroups of $G={\rm Sp}(W)$ containing $u$ are $G$-conjugate, then
$u$ has Jordan block structure $(J_a+J_1^r)$, with $a$ even, $a < p$, or $(J_a^2+J_1^r)$, for some 
$a$ odd, $a\leq p$.

For $G$ of type $B_n$ or $D_n$, a similar reasoning reduces our considerations to elements whose Jordan block decomposition is of the form $(J_a+J_1^r)$,
 for some $a$ odd,
$a\leq p$, or of the form $(J_a^2+J_1^r)$, for some $a$ even, $a < p$.

We record the conclusions of this discussion in the following
\begin{lemma}\label{classical_reduction} Let $G$ be one of the simple algebraic groups ${\rm Sp}(W)$ or ${\rm SO}(W)$ defined over a field of
  characteristic $p>2$. Let $u\in G$ be a non-identity unipotent element. If any two $A_1$-subgroups of $G$ containing $u$ are $G$-conjugate then one of the following holds.
\begin{enumerate}[(a)]
\item $G = {\rm Sp}(W)$ and $u$ has Jordan block decomposition on $W$ of the form $(J_a+J_1^r)$, with $a$ even, $a < p$, or $(J_a^2+J_1^r)$ with $a$ odd, $a\leq p$; or
\item $G = {\rm SO}(W)$ and $u$ has Jordan block decomposition on $W$ of the form $(J_a+J_1^r)$, with $a$ odd, $a\leq p$, or $(J_a^2+J_1^r)$ with $a$ even, $a< p$.
\end{enumerate}

\end{lemma}

In the following two propositions we establish Theorem \ref{t:mainclassical} for $G$ of type $B_n, C_n$ and $D_n$.

   \begin{prop} Let $G = {\rm Sp}_{2m}(k) = {\rm Sp}(W)$, $m\geq 2$, $W$ defined over $k$ of characteristic $p>2$, and let $u\in G$ be unipotent of order 
$p$. Then  any two $A_1$-subgroups of $G$ containing $u$ are $G$-conjugate
    if and only if the Jordan block decomposition of $u$ acting on $W$ is of the form $(J_{a}+J_1^r)$, for $a$ even, $a < p$, or
    $(J_a^2+J_1^r)$,
    for $a$ odd, $a<p$. 

\end{prop}

\begin{proof} We treat the possible Jordan block decompositions for $u$ given in Lemma~\ref{classical_reduction}(a).

  If $u$ acts as $(J_{2m})$, then $u$ is regular in $G$ and the result follows by Proposition \ref{p:regular}.

  Suppose now that $u$ acts on $W$ as $(J_{a}+J_1^r)$, with $a$ even, $a<p$, and $r>0$.  Then by \cite[Corollary 5.8]{K2}, any $A_1$-subgroup of $G$ containing $u$
  acts completely reducibly on
$W$. Now we apply Lemmas \ref{twofold_tensors} and \ref{threefold_tensors} and  see that there is a unique representation $\rho:{\rm SL}_2(k)\to G$ (up to twists and equivalence)
 with $u\in {\rm im}(\rho)$.

 Suppose now that $u$ acts on $W$ with Jordan blocks $(J_a^2+J_1^r)$, $a$ odd, $a\leq p$. If $a<p$, we  argue as in the previous paragraph;
 that is, given that $a$ is
odd, up to equivalence
and
 Frobenius twists, the only representation $\rho:{\rm SL}_2(k)\to G$ with $u\in{\rm im}(\rho)$ is given by the action of ${\rm SL}_2(k)$ on 
 $L(a-1)\oplus L(a-1)^*+N$, where $N$ is an $r$-dimensional trivial module.

 Finally, consider the case where $u$ acts on $W$ with Jordan block decomposition $(J_p^2+J_1^r)$. Here, we exhibit two non-conjugate $A_1$-subgroups of $G$
 each
 containing $u$. One such can be obtained as in the previous paragraph via the action of ${\rm SL}_2(k)$ on the module  $L(p-1)\oplus L(p-1)^*+N$, $N$ an
 $r$-dimensional
 trivial module. The second action is that of ${\rm SL}_2(k)$ on the module $(L(1)\otimes L(p-1)^F)\oplus N$ (see Lemma~\ref{twofold_tensors}).\end{proof}

\begin{prop} Let $G = {\rm SO}(W)$ where $W$ is a $2m$ or $(2m+1)$-dimensional $k$-vector space, $\dim W\geq 7$, defined over 
  $k$ of characteristic $p>2$, and let $u\in G$ be unipotent of order $ p$. Then any two $A_1$-subgroups of $G$ are $G$-conjugate
 if and only if the Jordan block decomposition of $u$ acting on $W$ is of the form $(J_a+J_1^r)$,  for $a$ odd, $a\leq p$ and
  $a\ne 3$ if $r>0$, or
    $(J_a^2+J_1^r)$,
    for $a$ even, $a<p$. 

\end{prop}

\begin{proof} We treat the  possible Jordan block decompositions for $u$ given in Lemma~\ref{classical_reduction}(b).

If $u$ acts as $(J_{2m-1}+J_1)$ or as $(J_{2m+1})$ on $W$, then $u$ is regular in $G$ and the result follows by Proposition \ref{p:regular}.

Let $\text{dim}(W)= r+4$, $r\geq 0$. We now consider the representation $\rho: {\rm SL}_2(k) \to G$ given by the action of ${\rm SL}_2(k)$ on the module $(L(1)\otimes L(1)^F)+N$, for $N$ an
$r$-dimensional
trivial module. The Jordan block decomposition of a non-identity unipotent element in ${\rm SL}_2(k)$ on such a module is $(J_3+J_1^{r+1})$. In addition, 
the 
representation $\rho': {\rm SL}_2(k) \to G$ afforded by the module $L(2)+N'$, for an $(r+1)$-dimensional trivial module $N'$ also has image meeting the class of 
elements with Jordan block decomposition $(J_3+J_1^{r+1})$. Thus, there exist non-conjugate $A_1$-subgroups containing $u$.

Suppose now that $u$ acts on $W$ as $(J_a+J_1^r)$, with $a$ odd, $a\leq p$, $a\ne 3$,  and $r>0$.
If $a<p$, then by \cite[Corollary 5.8]{K2}, any $A_1$-subgroup of $G$ containing $u$ acts completely reducibly on $W$ and  Lemmas ~\ref{twofold_tensors} 
and ~\ref{threefold_tensors} show that
up to twists and equivalence, there is a unique representation $\rho:{\rm SL}_2(k)\to G$ with $u\in {\rm im}(\rho)$, namely given by the action of
 ${\rm SL}_2(k)$ on 
$L(a-1)+N$, for $N$ an $r$-dimensional trivial module. Now if $a=p$, so that $u$ acts as $(J_p+J_1^r)$, for $r>1$, we may no longer apply 
\cite[Corollary 5.8]{K2}. (The case $r\leq 1$ was covered in the first paragraph of the proof.) Let $X$ be 
an $A_1$-subgroup of $G$ containing $u$. Write $W$ as a sum of indecomposable $X$-modules. On any reducible indecomposable summand, $u$ has a unique 
non-trivial Jordan block; 
then \cite[Prop. 5.7]{K2} implies that such a summand is not self-dual and so the dual of the summand must also occur as a summand and this is not 
compatible with the Jordan 
block structure of $u$. Hence once again, $X$ acts completely reducibly on $W$ and we may proceed as for the case $a\not\in\{p,3\}$ to see that 
$X$ is unique up to 
conjugacy in $G$. 

Suppose now that $u$ acts on $W$ as $(J_{a}^2+J_1^r)$, with $a$ even, $a<p$.  Then by \cite[Corollary 5.8]{K2}, any $A_1$-subgroup of $G$ 
containing $u$ acts completely reducibly on
$W$, and  applying Lemmas~\ref{twofold_tensors} and \ref{threefold_tensors} we 
find that there is a unique (up to twists and equivalence) representation  $\rho:{\rm SL}_2(k)\to G$ with $u\in {\rm im}(\rho)$, 
namely arising from the action of
${\rm SL}_2(k)$ on $L(a-1)\oplus L(a-1)^* + N$, where $N$ is an $r$-dimensional trivial module.\end{proof}

\section{Exceptional Groups}\label{sec:exc}

In this section we will treat the simple groups of exceptional type, proving Theorem~\ref{t:mainexceptional}. We first recall some facts about the
notation we will use for identifying unipotent conjugacy classes in $G$. Since  $p$ is a good prime for $G$, we will use the parametrisation of unipotent
conjugacy classes provided by  Bala--Carter--Pommerening in \cite{BCI,BCII,PomI, PomII}. There is a bijection between $G$-conjugacy classes of unipotent
elements
and $G$-classes of pairs $(L,P_L)$, where $L$ is a Levi subgroup of $G$ and $P_L$ is a distinguished parabolic subgroup of $[L,L]$. The $G$-class
corresponding to a
    pair $(L,P_L)$ contains the dense orbit of $P_L$ acting on its unipotent radical, which meets a class of so-called ``distinguished" unipotent elements.
  We will only be concerned with unipotent classes in the exceptional groups of type $E_n$. (The results of \cite{LaTeA1} will allow us to reduce to $G$ of type $E_n$. See Propositions~\ref{p:largeprimes} and ~\ref{reduction}.) More specifically those which
    have Bala--Carter--Pommerening ``label'' $\Psi$, $(\Psi)'$, $(\Psi)''$,  or $\Psi(a_i)$, where one of the following holds:
    \begin{enumerate}[i.]
    \item $\Psi$ is a root system and the associated class is the regular element in  a Levi factor of type $\Psi$ in $G$ and all such Levi factors
      are conjugate;
      \item There are non-conjugate Levi factors of root system type $\Psi$ and the classes labelled $(\Psi)'$ and $(\Psi)''$ correspond to regular
        elements in a
        representative  of each of the two  classes; or
      \item $\Psi$ is of type $E_m$ or $D_m$ and the class is labelled by $\Psi(a_i)$ for some $i$, which corresponds to a specific class of distinguished
        unipotent elements in a Levi factor of type $\Psi$ in $G$.
\end{enumerate} 

Now in addition to the above discussion, we  remind the reader that the class labelling is consistent with taking $L$-classes in Levi factors
$L$; that is,
if a class in $L$ has label $\Psi$ or $\Psi(a_i)$, then the class in $G$ has the same label, with two exceptions (in our setting): \begin{itemize}
  \item There is a unique class
    of regular elements in each Levi factor, and there are non-conjugate Levi factors of this same root system type. Then we require two labellings,
    for example in
    $G = E_7$, we have two classes labelled $(A_5)'$ and $(A_5)''$ for the non-conjugate Levi factors of type $A_5$ and the two $G$-classes meeting the
    class of regular
elements in a representative of each class of Levi factors of type $A_5$. (See \cite[p.272]{Law_unimax} for a more detailed explanation.)
\item In the groups of type $D_m$, we must distinguish the Levi factors of type $D_3$ from
  those of type $A_3$ (even though the two root systems are isomorphic), while in $G$ of type $E_n$, we view all Levi factors with root system $A_3$ as
  being of type $A_3$.
  \end{itemize}

  This discussion is particularly relevant with regards to our use of the tables in \cite{Law_unimax}. In this paper, Lawther considers the fusion of
  unipotent classes
  of  maximal connected reductive subgroups of the exceptional groups. The results are presented in the form of tables whose rows correspond to the
  classes in
  the maximal subgroup and whose columns give the labelling of the $G$-class. When a class is not mentioned in the table, we are meant to understand
  that the labelling
in the ambient group is the same as in the maximal subgroup. So for example, for the maximal subgroup $D_8$ in $G = E_8$, covered in
\cite[\S4.13, Table 22]{Law_unimax}, the $D_8$-class labelled $A_6$ does not appear in the first column, so we are to understand that the elements in
this class
lie in the $E_8$-class labelled $A_6$.

Let $N(A_1,G)$ be as defined in \cite[p.2]{LiSe96}. For $p>N(A_1,G)$ Theorem \ref{t:mainexceptional} can be established using the results of \cite{LaTeA1}; we record this in the following proposition.

\begin{prop} \label{p:largeprimes}
Let $G$ be an exceptional type simple algebraic group defined over a field of good characteristic $p>N(A_1,G)$ and let
  $u \in G$ be unipotent of order $p$. Then  any two $A_1$-subgroups of $G$ containing $u$ are $G$-conjugate if and only if $\class(u)$ is one of the following:
  \begin{enumerate}[(a)]
  \item $G = G_2$ and $\class(u)\in\{A_1, \tilde{A_1}, G_2\}$;
  \item $G = F_4$ and $\class(u)\in\{A_1, \tilde{A_2}, B_2, B_3, C_3, F_4(a_1), F_4\}$;
  \item $G = E_6$ and $\class(u)\in\{A_1, A_3, D_4, A_5, D_5, E_6(a_1), E_6\}$;
  \item $G = E_7$ and $\class(u)\in\{A_1, A_3, D_4, (A_5)'', (A_5)', D_5,A_6, D_6, E_6(a_1), E_6, E_7(a_1), E_7\}$;
  \item $G = E_8$ and $\class(u)\in\{A_1, A_3, D_4, A_5, D_5, E_6(a_1), D_6, E_6, A_7, D_7, E_7(a_1), E_7, E_8(a_4),$

    $ E_8(a_2), E_8(a_1), E_8\}$.
\end{enumerate}
\end{prop}

\begin{proof}
For $p>N(A_1,G)$, Theorems 1 and 4 of \cite{LaTeA1} determine the classes of
$A_1$-subgroups containing a representative of each unipotent conjugacy class. More precisely, an $A_1$-subgroup is determined up to $G$-conjugacy by its  (uniquely determined) non-negatively labelled diagram.
Under this restriction on $p$, we then use Tables 1 - 5 of \cite{LaTeA1} to deduce the result.
\end{proof}

Henceforth, we will assume that $p\leq N(A_1,G)$, while still assuming that $p$ is good for $G$. Thus we must consider the pairs $(G,p)\in\{(E_6, 5),
(E_7, 5),
(E_7,7), (E_8,7)\}$.

\begin{prop}\label{reduction} Let $(G,p)\in\{(E_6, 5),(E_7, 5),
(E_7,7), (E_8,7)\}$ and let
  $u \in G$ be unipotent of order $p$. Suppose that any two $A_1$-subgroups of $G$ containing $u$ are $G$-conjugate. Then $\class(u)$ is one of the following:
  \begin{enumerate}[(a)]
  \item $(G,p) =  (E_6,5)$ and $\class(u)\in\{A_1, A_3 \}$;
  \item $(G,p) = (E_7,5)$ and $\class(u)\in\{A_1, A_3 \}$;
  \item $(G,p) = (E_7,7)$ and $\class(u)\in\{A_1, A_3, D_4, (A_5)'', (A_5)', A_6 \}$;
  \item $(G,p) = (E_8,7)$ and $\class(u)\in\{A_1, A_3, D_4, A_5\}$.
    \end{enumerate}
  \end{prop}

  \begin{proof} 
    We rely upon \cite[Tables 5,7,9]{Law_jordan} to determine which unipotent elements have order $p$. In addition, by Proposition~\ref{nonsimple}, we may assume that $u$ is distinguished in a Levi factor $L$ with $[L,L]$ simple. We treat in turn the remaining
    cases which
    do not appear in the statement of the result.

    Consider first the case $G = E_6$ and $p=5$. Then $u$  distinguished in a Levi factor $L$ as described above, of order $5$ and $\class(u)$ not listed in (a) imply that $\class(u)\in\{A_2,A_4, D_4(a_1)\}$.
    If $\class(u) = A_2$, we note that the $4$-dimensional representations of ${\rm SL}_2(k)$ with highest weight $1+p^a$ for different values of
    $a\geq 1$, have non-conjugate
    images in $A_3$ and each contains a fixed element whose class in $A_3$ is the $A_2$-class. The $A_1$-subgroups of $G$ obtained by taking such
    embeddings in a
    fixed $A_3$-Levi subgroup of $G$ will give rise to non-conjugate $A_1$-subgroups of $G$ containing $u$ (seen by considering the action on ${\rm Lie}(G)$). 

    For $\class(u) = A_4$, we consider embeddings of $A_1$ in $D_5$ which factor through a $B_2B_2$-subgroup of $D_5$. Then
    \cite[Theorem 3.1, Lemma 3.11]{LSbook} show that this embedding meets the $A_4$-class of $D_5$, and by taking a $D_5$-parabolic subgroup of $G$,
    will also lie in the $A_4$-class of $G$. Now we apply
    Proposition~\ref{nonsimple} to see that $u$ lies in non-conjugate $A_1$-subgroups of $G$.

    Finally, for the $D_4(a_1)$-class, we use the fact that this class meets a $B_2B_1$-subgroup of $D_4$ and again apply Proposition~\ref{nonsimple}.

    Turn now to the case $G = E_8$ and $p=7$. Then $u$ of order $7$ distinguished in a Levi factor $L$ and $\class(u)$ not listed in the statement of the result implies that
    $\class(u)\in\{A_2,A_4, D_4(a_1),$
    $D_5(a_1), A_6, E_6(a_3), D_6(a_2), E_7(a_5), E_8(a_7)\}$. For the first two classes, argue precisely as for
    $G=E_6$ when $p=5$.
    Each class of the form $D_m(a_j)$ is represented in a $B_aB_d$-subgroup of a $D_m$-Levi factor of $G$, and we apply Proposition~\ref{nonsimple}.
    If $\class(u) = A_6$, we note that this class meets a $B_3B_3$-subgroup of $D_7$ and conclude by using Proposition~\ref{nonsimple}.

    For the class $E_6(a_3)$, we refer to \cite[\S4.8, Table 17]{Law_unimax} to see that the class intersects non-trivially the class of regular
    elements in
    an $A_1A_5$-subgroup of $G$, and then apply Proposition~\ref{nonsimple}. The argument is entirely similar for the $E_7(a_5)$-class
    which meets a class of distinguished elements in a $D_6A_1$-subsystem subgroup of $G$ (see \cite[\S4.10, Table 19]{Law_unimax}) and the $E_8(a_7)$-class meets
    the class of regular elements in an $A_4A_4$-subgroup of $G$ (see \cite[\S4.17, Table 26]{Law_unimax}).

    Now turn to $G = E_7$. For $p=5$, the classes requiring consideration are precisely as in $E_6$ and the arguments given there are valid here as well.
    For $p=7$, we must consider the classes $\{A_2, A_4, D_4(a_1), D_5(a_1), D_6(a_2), E_6(a_3), E_7(a_5)\}$.
    Each of these can be treated as in $E_8$.
\end{proof}

Before starting the proof of Theorem \ref{t:mainexceptional}, we recall one further set of definitions: a subgroup $H$ of $G$ is said to be \emph{$G$-completely reducible} ($G$-cr) if
whenever
$H \subseteq P$ for some parabolic subgroup $P$ of $G$, then $H$ is contained in a Levi subgroup of $P$. If the subgroup $H$ lies in no proper
parabolic subgroup
of $G$, we say that $H$ is \emph{$G$-irreducible} (and \emph{$G$-reducible} if $H$ lies in a proper parabolic subgroup of $G$). Given a subgroup
$H$ of $G$, we have that exactly one of the following must hold: $H$ is non-$G$-cr; $H$ is $G$-irreducible; or $H$ is $G$-cr and $G$-reducible. The proof of Theorem \ref{t:mainexceptional} follows by treating the $A_1$-subgroups of each of these types in turn. 

We now assume that $\class(u)$ is as in Proposition~\ref{reduction} and $p\leq N(A_1,G)$. We first investigate the existence or non-existence of a
non-$G$-cr $A_1$-subgroup of $G$ meeting $\class(u)$. We will establish the following.

\begin{prop} \label{p:nonGcr}
  Let $G$ be a simple algebraic group of type $E_n$, and assume $p$ is good for $G$ and $p\leq N(A_1,G)$. Let $u\in G$ of order $p$ and
  $X\subseteq G$ a non-$G$-cr $A_1$-subgroup containing $u$. Then \begin{enumerate}[(a)]
  \item $\class(u)\not\in\{ A_3,D_4\}$;
   \item $\class(u)\not\in\{(A_5)',
      (A_5)''\}$ if $G = E_7$; and 
  \item $\class(u) \ne A_5$ if $G = E_8$.
   
    \end{enumerate}

    Moreover, when $p=7$, there exist two non-conjugate non-$G$-cr $A_1$-subgroups of $G=E_7$ which contain a unipotent element $u$ with $\class(u) = A_6$. 

\end{prop}

\begin{proof} Here we will rely on the tables in \cite{LiTh_CR} (and the corrected version of \cite[Tables 11,12]{LiTh_CR} in ~\cite{LiTh_CR_arxiv}), as well as the known Jordan block structure on
  certain Weyl modules $W(c)$ and tilting modules $T(c)$ for ${\rm SL}_2(k)$, as discussed in \S2.
  Finally we compare this information with the tables in \cite{Law_jordan}.

  Consider first the case $G = E_6$ and $p=5$.
  We must determine whether any of the non-$G$-cr $A_1$-subgroups appearing in \cite[Table 11]{LiTh_CR_arxiv} meet the $A_3$-class of unipotent elements.
  Using the Jordan block structure of the tilting modules, it follows that each non-$G$-cr $A_1$-subgroup has unipotent elements with at least 2 Jordan blocks
  of size $5$ in the action on the 27-dimensional
  irreducible $G$-module and then \cite[Table 5]{Law_jordan} shows that such elements do not lie in the $A_3$-class.
  The case $G = E_7$ and $p=5$ is entirely similar; we compare the actions of non-$G$-cr $A_1$-subgroups described in \cite[Table 12]{LiTh_CR_arxiv} with
  the information in \cite[Table 7]{Law_jordan}.

  Now turn to the case $G=E_7$ and $p=7$, where we must consider the classes $A_3, D_4$, $(A_5)', (A_5)''$ and $A_6$. For all but the last class, there
  are
  at most 6 blocks of size 7 in the action on the $56$-dimensional module (see \cite[Table 7]{Law_jordan}) while \cite[Table 13]{LiTh_CR} implies 
  that the classes of non-$G$-cr $A_1$-subgroups contain elements who Jordan block structure on the 56-dimensional module is $(J_7)^8$. This completely
  identifies
  the class of the non-identity unipotent elements in each of these $A_1$-subgroups as the $A_6$-class (using again \cite[Table 7]{Law_jordan}). So here we have two $G$-conjugacy classes of non-$G$-cr $A_1$-subgroups meeting the $A_6$-class as claimed.

  It remains to consider the case of $G = E_8$ when $p=7$ and the three classes $A_3, D_4$ and $A_5$. Here \cite[Table 9]{Law_jordan} shows that a unipotent element in any of these three classes acts with at most 28 Jordan blocks of size $7$ on  ${\rm Lie}(G)$.
 Now  using the structure of tilting and Weyl
  modules, and
  the decompositions given in \cite[Table15]{LiTh_CR}, we work out that a non-$G$-cr $A_1$-subgroup in $G$ has unipotent elements with at least
  29 Jordan blocks of size 7 in the action on ${\rm Lie}(G)$.

    This completes the proof of the proposition.\end{proof}

We now investigate the existence or non-existence of a
$G$-irreducible $A_1$-subgroup of $G$ meeting $\class(u)$. We will establish the following.

\begin{prop} \label{prop:G-irr}
Let $G$ be a simple algebraic group of type $E_n$, and assume $p$ is good for $G$ and $p\leq N(A_1,G)$. Let $u\in G$ of order $p$ and $X\subseteq G$ a $G$-irreducible $A_1$-subgroup containing $u$. Then \begin{enumerate}[(a)]
  \item $\class(u)\not\in\{ A_3,D_4\}$;
  \item $\class(u)\not\in\{(A_5)',
      (A_5)''\}$ if $G = E_7$; and
  \item $\class(u) \ne A_5$ if $G = E_8$.

    \end{enumerate}
\end{prop}

\begin{proof}
  We make use of the tables in \cite{Thomas_A1}, where each $G$-irreducible $A_1$-subgroup is identified as lying in a reductive, maximal connected subgroup of $G$, and we identify the unipotent class meeting the $G$-irreducible $A_1$-subgroups thus described. We then compare this to the tables given in \cite{Law_unimax} to determine the unipotent class in $G$. 
  In order to determine the class in the maximal connected subgroups, we consider the embedding given by \cite{Thomas_A1} and calculate the Jordan
  block structure of the unipotent action. 
Calculations of the Jordan block structure of tensor products follow from \cite[Theorem 1]{Renaud}.

Consider first $G=E_6$ and $p=5$. We determine whether any of the $G$-irreducible $A_1$-subgroups appearing in \cite[Table 6]{Thomas_A1} meet the
$A_3$-class of unipotent elements. 
We first take $X$ to be contained in $\bar{A_1}A_5$. 
In considering the projection of $X$ into the  $A_5$-factor, we see that the Jordan block structure of non-identity unipotent elements of $X$ associated
to this embedding is
given by $(J_4+J_2)$. 
It is then clear that the unipotent class in the $A_5$-component is not in the $A_3$-class.
By \cite[\S4.8, Table 17]{Law_unimax} we conclude that $X$ does not meet the $A_3$-class in $G$. 
Next let $X$ be contained in the maximal subgroup $A_2^3$. By \cite[\S4.9, Table 18]{Law_unimax} we see that $X$ does not meet the $A_3$-class in $G$.

Now consider $G=E_7$ and $p=5$. 
We determine whether any of the $G$-irreducible $A_1$-subgroups appearing in \cite[Table 7]{Thomas_A1} meet the $A_3$-class of unipotent elements. 
First take $X$ to be contained in the maximal subgroup $\bar{A_1}D_6$. By \cite[\S4.10, Table 19]{Law_unimax} we have that $X$ meets the class $A_3$
if it meets either the $A_3$ or $D_3$-class in $\bar{A_1}D_6$. 
We find that each embedding in \cite[Table 7]{Thomas_A1} has one of the following Jordan block structures in the $D_6$ component: $(J_5+J_3^2+J_1), (J_3^4), (J_3^3+J_1^3)$, and observe that none of these corresponds to an $A_3$ or $D_3$-class. 
Now let $X$ be contained in the maximal subgroup $A_1A_1$. By \cite[\S5.8, Table 34]{Law_unimax} we have that $X$ does not meet the $A_3$-class in $G$. 
 
Next consider $G=E_7$ and $p=7$, where we consider the classes $A_3,D_4,(A_5)'$ and $(A_5)''$.
Once again we refer to \cite[Table 7]{Thomas_A1} and determine when the given $G$-irreducible $A_1$-subgroup meets the stated classes of unipotent
elements. 
First take $X$ to be contained in $\bar{A_1}D_6$. As in the above paragraph, by \cite[\S4.10, Table 19]{Law_unimax}, in order to obtain the $A_3$-class in $G$
we require the embedding in $\bar{A_1}D_6$ to meet either the $A_3$ or $D_3$-class in $D_6$, and the class labelling remains the same for the other
classes. 
We find that each embedding in \cite[Table 7]{Thomas_A1} has one of the following Jordan block structures in the $D_6$ component: 
$(J_7+J_5),(J_5+J_3^2+J_1),(J_7+J_3+J_1^2),J_3^4,(J_3^3+J_1^3)$.
It is clear that none of these partitions corresponds to any of the $D_6$ classes $A_3,D_3,D_4$ or $A_5$ and so we are done.

Now let $X$ be contained in a maximal subgroup $G_2C_3$. The Jordan block structure in the $C_3$ component is given by either $(J_6)$ or $(J_4+J_2)$ which correspond to the $C_3$ and $C_3(a_1)$ classes, respectively. 
By \cite[\S5.12, Table 38]{Law_unimax} we have that $X$ cannot meet any of the given classes in $G$. 
Finally let $X$ be contained in either the maximal subgroup $A_1G_2$ or $A_1A_1$. By \cite[\S5.9, Table 35]{Law_unimax}, respectively \cite[\S5.8, Table 34]{Law_unimax}, we have that $X$ does not meet the given classes in $G$.

Finally consider $G=E_8$ and $p=7$, where
we must consider the classes $A_3,D_4$ and $A_5$. 
We proceed as above, referring here to \cite[Table 8]{Thomas_A1}. 
First take $X$ to be contained in the maximal connected subgroup $D_8$. 
By \cite[\S4.13, Table 22]{Law_unimax} we see that $X$ meets the $A_3$-class of $G$ when $X$ meets the $A_3$ or $D_3$-class of $D_8$; the class
labelling remains the same for other pertinent classes. 
We find that each embedding in \cite[Table 8]{Thomas_A1} has one of the following Jordan block structures in $D_8$:
$(J_7+J_5+J_3+J_1), (J_5+J_3^3+J_1^2), (J_7+J_3^2+J_1^3), (J_7+J_3^3), (J_5^3+J_1), (J_5^2+J_3^2), (J_3^5+J_1)$ and $(J_3^4+J_1^4)$.
It is clear that none of these partitions corresponds to a regular element in an $A_3,D_3,D_4$ or $A_5$ Levi subgroup  of $D_8$ and so we are done.

Next let $X$ be contained in a maxiaml subgroup $A_1E_7$. 
We consult \cite[\S4.14, Table 23]{Law_unimax} and observe that we cannot find $X$ meeting any of the given classes unless we have the same class
labelling in $A_1E_7$.
We observe from \cite[Table 8]{Thomas_A1} that each embedding into the $E_7$ component has been considered in the $E_7$ case above. We saw that $X$
did not meet any of the given classes in the above argument, and hence this holds in $G$.
Finally, let $X$ be contained in a maximal subgroup $G_2F_4$. Here we note that the embedding $G_2(\#3)$ corresponds to the regular class in $G_2$ and $F_4(\#11)$
corresponds to the class $F_4(a_2)$.  By \cite[\S5.12, Table 38]{Law_unimax}, we have that $X$ does not meet any of the given $G$-classes. 
\end{proof}

We now assume that $\class(u)$ is as in Theorem~\ref{t:mainexceptional}, of order $p\leq N(A_1,G)$. We demonstrate that there is a unique class of $G$-cr,
$G$-reducible $A_1$-subgroups of $G$ meeting $\class(u)$.
The following two lemmas will be useful.

\begin{lemma} \label{lemma:An}
  Let $L$ be a simple algebraic group of type $A_n$. Let $X \subseteq L$ be an $L$-irreducible $A_1$-subgroup with $u \in X$ unipotent, $u \neq 1$
  such that $u$ acts on $X$ with Jordan block structure $(J_\ell+J_1^r)$ for $\ell \notin \{3,p\}$. Then $\ell = n$, $r = 0$, and $X$ is the unique 
(up to $G$-conjugacy) $A_1$-subgroup of $L$ containing $u$.
\end{lemma}

\begin{proof}
  It follows from Proposition \ref{prop:SLn} that there is a unique class of $A_1$-subgroups containing $u$ of this form. This can be obtained by the
  action of ${\rm SL}_2(k)$ on $L(\ell -1) \oplus N$, where $N$ is an $r$-dimensional trivial module. Hence $X$ is $L$-irreducible if and only if
  $\ell = n$ and $r=0$.
\end{proof}

\begin{lemma} \label{lemma:Dn}
  Suppose $p \in \{5,7\}$ and let $L$ be a simple algebraic group of type $D_n$ for $4 \leq n \leq 7$, and take $W$ to be the natural module for $L$.
  Let $X \subseteq L$ be an $A_1$-subgroup with $u \in X$ unipotent, $u \neq 1$, such that 
    $W \downarrow X = W_1 \perp \dots \perp W_k$, where $W_i$ is non-degenerate, irreducible for all $i$ and the $W_i$ are pairwise inequivalent as $X$-modules. 
    Then the Jordan block decomposition of $u$ on the natural module is one of the following: 
    \begin{enumerate}[(a)]
        \item $L = D_4$: 
        \[(J_5+J_3),(J_3^2+J_1^2) \ \text{for} \  p=5, \ \text{and}\]
        \[(J_7+J_1), (J_5+J_3),(J_3^2+J_1^2) \ \text{for} \ p=7;\]
        \item $L = D_5$: 
        \[(J_5^2),(J_5+J_3+J_1^2),(J_3^3+J_1) \ \text{for} \ p=5, \ \text{and} \] 
        \[(J_7+J_3),(J_5^2),(J_5+J_3+J_1^2),(J_3^3+J_1) \  \text{for} \ p=7;\]
        \item $L=D_6$: 
        \[(J_5+J_3^2+J_1),(J_3^3+J_1^3), (J_3^4) \ \text{for} \ p=5, \ \text{and}\] 
        \[(J_7+J_5),(J_7+J_3+J_1^2), (J_5+J_3^2+J_1),(J_3^3+J_1^3), (J_3^4) \ \text{for} \  p=7; \ \text{and}\] 
        \item $L=D_7$: 
        \[(J_5^2+J_3+J_1),(J_5+J_3^3),(J_5+J_3^2+J_1^3),(J_3^4+J_1^2) \ \text{for} \ p=5, \ \text{and} \]
        \[(J_7^2),(J_7+J_3^2+J_1),(J_5^2+J_3+J_1),(J_5+J_3^3),(J_5+J_3^2+J_1^3),(J_3^4+J_1^2) \ \text{for} \ p=7.\]
    \end{enumerate}
\end{lemma}

\begin{proof}
  First observe that since each $W_i$ is non-degenerate and irreducible, $W$ is either irreducible, or a direct sum of inequivalent irreducible
  ${\rm SL}_2(k)$-modules on each of which ${\rm SL}_2(k)$ preserves a non-degenerate orthogonal form. The possible such ${\rm SL}_2(k)$-modules are
  given in Table \ref{Table1}  (up to a permutation of Frobenius twists). Consideration of the direct sums of these ${\rm SL}_2(k)$-modules is enough to
  prove the lemma. 
    \begin{table}[ht!]
\centering
\begin{tabular}{| m{5cm} | m{2.5cm} | m{2.5cm} | m{2.5cm} |}  
 \hline
  $W_i$ & Dimension & Jordan block form for $p=5$ & Jordan block form for $p=7$ \\ 
 \hline\hline
 $L(0)$ & 1 & $(J_1)$ & $(J_1)$ \\
 $L(2)$ & 3 & $(J_3)$ & $(J_3)$ \\
 $L(1)\otimes L(1)^F$ & 4 & $(J_3+J_1)$ & $(J_3+J_1)$ \\
 $L(4)$ & 5 & $(J_5)$ & $(J_5)$ \\
 $L(6)$ & 7 &  & $(J_7)$ \\
 $L(1)\otimes L(3)^F$ & 8 & $(J_5+J_3)$ & $(J_5+J_3)$ \\
 $L(2)\otimes L(2)^F$ & 9 & $(J_5+J_3+J_1)$ & $(J_5+J_3+J_1)$ \\
 $L(1)\otimes L(1)^F\otimes L(2)^{F^r}, r>1$ & 12 & $(J_5+J_3^2+J_1)$ & $(J_5+J_3^2+J_1)$ \\
 $L(1)\otimes L(5)^F$ & 12 &  & $(J_7+J_5)$ \\
 
 \hline
\end{tabular}
\caption{Irreducible ${\rm SL}_2(k)$-modules with an orthogonal form, and dimension at most 14}
\label{Table1}
\end{table} 
\end{proof}

\begin{prop} \label{p:Greducible}
  Let $G$ be a simple algebraic group of type $E_n$, and assume $p$ is good for $G$ and $p \leq N(A_1,G)$. Let $u \in G$ be of order $p$, with
  $\class(u)$ one of the following:
\begin{enumerate}[(a)]
  \item $G = E_6, p=5$ and $\class(u) =  A_3$;
  \item $G = E_7, p=5$ and $\class(u)= A_3$;
  \item $G = E_7,p=7$ and $\class(u)\in \{ A_3, D_4, (A_5)',(A_5)''\}$; or
  \item $G=E_8, p=7$ and $\class(u)\in \{ A_3, D_4, A_5\}$.
\end{enumerate}
For $u$ in each of the given classes, any two $G$-cr, $G$-reducible $A_1$-subgroups of $G$ containing $u$ are $G$-conjugate.
\end{prop}

\begin{proof}
Let $X \subseteq G$ be a $G$-cr, $G$-reducible $A_1$-subgroup of $G$ containing $u$ with $\class(u)$ as in the statement of the proposition.  
We have that $X$ is $L$-irreducible in some minimal Levi subgroup $L$ of $G$. Note that $L$ is unique up to $G$-conjugacy by
\cite[Proposition 6.1]{LiTh_Red}, so it suffices to establish unicity up to conjugacy in $L$. By \cite[Lemma 2.2]{LieTe04} we see that one of the
following must hold: $L$ is of type $A_n$ and $X$ acts irreducibly on the natural $L$-module, $L$ is of type $D_n$ and $X$ acts on the natural
module as $W \downarrow X = W_1 \perp \dots \perp W_k$ for some $W_i$ non-degenerate, irreducible and inequivalent as $X$-modules, or $L$ is of
exceptional type.    

 First take $G=E_6, p=5$, and so $\class(u) =  A_3$. 
 Suppose that $L$ is of type $A_n$, $3 \leq n \leq 5$, so that $X$ meets the $A_3$-class if and only if $u$ acts with Jordan blocks $(J_4+J_1^r)$ on the
 natural $L$-module. 
 It follows from Lemma \ref{lemma:An} that this can only occur for $u$ acting as the regular element in $L=A_3$, in which case $X$ is unique up to
 conjugacy in $L$, by Proposition~\ref{p:regular}. 
 Next suppose that $L$ is of type $D_n$, $ n \in \{ 4,5 \}$. Recall that in groups of type $D_n$ we must distinguish the Levi factors of type $D_3$ from those of type $A_3$, hence $X$ meets the $A_3$-class if and only if $u$ acts with Jordan blocks $(J_4^2+J_1^r)$ or $(J_5+J_1^r)$
 on the natural $D_n$ module. By Lemma \ref{lemma:Dn}, we have that $X$ does not meet the $A_3$-class in $L$.

 Now take $G=E_7, p=5$, and so $\class(u) =  A_3$. 
 First suppose that $L$ is of type $A_3,A_4,(A_5)',$ $(A_5)''$ or $A_6$, so that $X$ meets the $A_3$-class if and only if $u$ acts with Jordan
 blocks $(J_4+J_1^r)$ on the natural $L$-module. 
 It follows from Lemma~\ref{lemma:An} that this can only occur for $u$ acting as the regular element in $L=A_3$ in which case $X$ is unique (up to $L$-conjugacy) by Proposition~\ref{p:regular}. 
 Next suppose that $L$ is of type $D_n$, $ 4 \leq n \leq 6$. Then $X$ meets the $A_3$-class if and only if $u$ acts with Jordan blocks $(J_4^2+J_1^r)$ or $(J_5+J_1^r)$
 on the natural $L$-module. 
 By Lemma \ref{lemma:Dn}, we have that $X$ does not meet the $A_3$-class in $L$. 
 Finally suppose that $L$ is of type $E_6$, then $X$ is $L$-irreducible and it follows from Proposition \ref{prop:G-irr} that $X$ does not meet
 the $A_3$-class. 

 Take $G=E_7, p=7$, and so $\class(u)\in \{ A_3, D_4, (A_5)',(A_5)''\}$. 
 First suppose that $L$ is of type $A_3,A_4,(A_5)',(A_5)'',A_6$, so that $X$ does not meet the $D_4$-class, and $X$ meets the $A_3, (A_5)'$ or $(A_5)''$-class if and only if $u$ acts with Jordan blocks $(J_4+J_1^r), (J_6+J_1^r)$ or $(J_6+J_1^r)$, respectively, on the natural $L$-module. 
 It follows from Lemma \ref{lemma:An} that these blocks can only occur for $u$ acting as the regular element in $L$ of type $A_3, (A_5)'$, or $(A_5)''$,
 respectively. 
 In each of these cases, again by Proposition~\ref{p:regular} there is a unique $L$-class of such $X$ which meet the given class.
 Next suppose that $L$ is of type $D_n$, $ 4 \leq n \leq 6$. Then $X$ meets the $A_3$, $D_4$, $(A_5)'$ or $(A_5)''$-class if and only if $u$ acts on $X$ with
 Jordan blocks $(J_4^2+J_1^r)$ or $(J_5+J_1^r)$, (for the $A_3$-class), or $(J_7+J_1^r), (J_6^2+J_1^r)$ or $(J_6^2+J_1^r)$ respectively (for the remaining classes)  on the natural $L$-module. By Lemma \ref{lemma:Dn}, the only case where
 such a Jordan block structure can occur is for the regular element in $D_4$, in which case $X$ is unique (up to $L$-conjugacy) as above. Finally suppose that $L$ is of type $E_6$,
 then $X$ is $L$-irreducible and it follows from Proposition \ref{prop:G-irr} that $X$ does not meet any of the given classes. 

 Finally, take $G=E_8, p=7$, and so $\class(u)\in \{ A_3, D_4, A_5\}$. 
 First suppose that $L$ is of type $A_n$, $3 \leq n \leq 7$, so that $X$ does not meet the $D_4$-class, and $X$ meets the $A_3$ or $A_5$-class if and
 only if $u$ acts with Jordan blocks $(J_4+J_1^r)$ or $(J_6+J_1^r)$ respectively on the natural $L$-module. 
 It follows from Lemma~\ref{lemma:An} that this can only occur for $u$ acting as the regular element in $L$ of type $A_3$ or $A_5$ respectively, in
 which case there is a unique $L$-class of such $X$ which meet the given class, as above. 
 Next suppose that $L$ is of type $D_n$, $4 \leq n \leq 7$. Then $X$ meets the $A_3$, $D_4$ or $A_5$-class if and only if $u$ acts on  on the natural $L$-module  with Jordan
 blocks $(J_4^2+J_1^r)$ or $(J_5+J_1^r)$ (for the $A_3$-class), or $(J_7+J_1^r)$ or $ (J_6^2+J_1^r)$, respectively (for the remaining classes). By Lemma \ref{lemma:Dn}, the only case where such a Jordan
 block structure can occur is for the regular element in $L= D_4$, in which case $X$ is unique (up to $L$-conjugacy) as above.
 Finally suppose that $L$ is of type $E_6$ or $E_7$, then $X$ is $L$-irreducible, and it follows from Proposition \ref{prop:G-irr} that $X$ does not
 meet any of the given classes. 
\end{proof}

We observe that the proof of the previous proposition shows that  if $X \subseteq G$ is a $G$-cr, $G$-reducible $A_1$-subgroup of $G$ meeting $\class(u)$, for $\class(u)$ as in Prop.~\ref{p:Greducible}, then $X$ must be $L$-irreducible in a
Levi subgroup $L$ of $G$, and the non-identity unipotent elements of $X$ are regular in $L$. 

We now combine the results of \S4 to prove Theorem \ref{t:mainexceptional}.

\begin{proof}[Proof of Theorem $\ref{t:mainexceptional}$]
The cases $G=G_2$, and $G=F_4$ are covered by Proposition~\ref{p:largeprimes}, as well as the cases $G=E_6$ for $p>5$, $G=E_7$ for $p>7$, and $G=E_8$ for $p>7$.
As we are assuming that $p$ is a good prime for $G$, we are left to consider the pairs $(G,p)\in\{(E_6, 5),
    (E_7, 5),
    (E_7,7), (E_8,7)\}$.
    In these cases we use Proposition \ref{reduction} to deduce that $u$ satisfies that any two $A_1$-subgroups of $G$ containing $u$ are $G$-conjugate only if $u$ is as
    in Theorem~\ref{t:mainexceptional}, or if $G=E_7$, $p=7$ and $u$ is in the $A_6$-class. 
Proposition~\ref{reduction} then establishes the existence of two non-conjugate non-$G$-cr $A_1$-subgroups meeting the $A_6$-class in $G=E_7$ when $p=7$.

It follows from Proposition \ref{p:longroot} that there is a unique class of $A_1$-subgroups of $G$ containing representatives of the $A_1$-class of unipotent elements. 

We now treat the remaining classes in the statement of Theorem \ref{t:mainexceptional}, i.e. those classes where $(G,p)$ is as above, $u$ has order $p$ and $\class(u)\ne A_1$.  Let $X$ be an $A_1$-subgroup of $G$ containing $u$, then it follows from Propositions \ref{p:nonGcr} and \ref{prop:G-irr} that $X$ is neither
non-$G$-cr nor $G$-irreducible.
Therefore $X$ must be $G$-cr and $G$-reducible, and then by Proposition~\ref{p:Greducible}, it follows that any two such $A_1$-subgroups containing $u$ are $G$-conjugate.
\end{proof}

\bibliographystyle{alpha}
\bibliography{rsch}

\end{document}